\documentclass[11pt]{article}

\usepackage{amsthm,amsfonts,amssymb,amsmath,color}
\usepackage{indentfirst}
\allowdisplaybreaks[4]

\numberwithin{equation}{section}

\renewcommand\d{\partial}

\newcommand\R{\mathbb R}

\def\O{\Omega}

\def\epsilon{\varepsilon}
\def\e{\varepsilon}

\newcommand\br{\begin{rem}}
\newcommand\er{\end{rem}}
\newcommand\bp{\begin{pmatrix}}
\newcommand\ep{\end{pmatrix}}
\newcommand\be{\begin{equation}}
\newcommand\ee{\end{equation}}
\newcommand\ba{\begin{equation}\begin{aligned}}
\newcommand\ea{\end{aligned}\end{equation}}

\newcommand\nn{\nonumber}

\usepackage{layout}

\newcommand{\uu}{{\mathbf u}}

\newcommand{\ff}{{\mathbf f}}

\newcommand{\dive}{{\rm div\,}}

\newtheorem{defi}{Definition}[section]
\newtheorem{theorem}[defi]{Theorem}
\newtheorem{proposition}[defi]{Proposition}
\newtheorem{lemma}[defi]{Lemma}
\newtheorem{corollary}[defi]{Corollary}
\newtheorem{remark}[defi]{Remark}

\numberwithin{equation}{section}

\begin{document}

\title{Homogenization of some evolutionary non-Newtonian flows in porous media}

\author{Yong Lu \footnote{Department of Mathematics, Nanjing University, Nanjing 210093, China, luyong@nju.edu.cn}\and Zhengmao Qian \footnote{Department of Mathematics, Nanjing University, Nanjing 210093, China, zmqian@smail.nju.edu.cn}}
\date{}

\maketitle

\begin{abstract}

In this paper, we consider the homogenization of evolutionary incompressible purely viscous non-Newtonian flows of Carreau-Yasuda type in porous media with small perforation parameter $0<\e\ll 1$, where the small holes are periodically distributed. Darcy's law is recovered in the homogenization limit.  Applying Poincar\'e type inequality in porous media allows us to derive the uniform estimates on velocity field, of which the gradient  is small of size $\e$ in $L^{2}$ space. This indicates the nonlinear part in the viscosity coefficient does not contribute in the limit and a linear model (Darcy's law) is obtained. The estimates of the pressure rely on a proper extension from the perforated domain to the homogeneous non-perforated domain. By integrating the equations in time variable such that each term in the resulting equations has certain continuity in time, we can establish the extension of the pressure by applying the dual formula with the restriction operator.

\end{abstract}

\renewcommand{\refname}{References}


\section{Introduction}

In this paper we consider the homogenization of evolutionary incompressible viscous non-Newtonian flows in porous media. Non-Newtonian fluids are extensively involved in a number of applied problems involving the production of oil and gas from underground reservoirs. There are at least two typical situations: the flow of heavy oils and enhanced oil recovery. Therefore, it is important to establish filtration laws governing non-Newtonian flows through porous media. In this paper, we will consider only quasi-Newtonian fluids where the viscosity can be expressed as a function of the shear rate. In particular, we focus on the  Carreau-Yasuda model in space-time cylinder $(0,T)\times\Omega_\varepsilon$:
\begin{equation}\label{1.3}
\begin{cases}
\varepsilon^2\partial_t\mathbf{u}_\varepsilon-{\rm div}\,\big(\eta_r(D\mathbf{u}_\varepsilon)D\mathbf{u}_\varepsilon \big)+(\mathbf{u}_\varepsilon\cdot\nabla)\mathbf{u}_\varepsilon
+\nabla p_\varepsilon=\mathbf{f},& {\rm in}\;(0,T)\times\Omega_\varepsilon, \\
{\rm div} \,\mathbf{u}_\varepsilon=0,& {\rm in}\;(0,T)\times\Omega_\varepsilon,\\
\mathbf{u}_\varepsilon=0,& {\rm in}\;(0,T)\times\partial\Omega_\varepsilon,\\
\mathbf{u}_\varepsilon|_{t=0}=\mathbf{u}_0,&{\rm in}\;\Omega_\varepsilon.
\end{cases}
\end{equation}
Here $\mathbf{u}_\varepsilon$ is the velocity, $\nabla \mathbf{u}_\varepsilon$ is the gradient velocity tensor, $D\mathbf{u}_\varepsilon=\frac{1}{2}(\nabla \mathbf{u}_\varepsilon+\nabla^T\mathbf{u}_\varepsilon)$ denotes the rate-of-strain tensor, $p_\varepsilon$ denotes the pressure, $\mathbf{f}$ is the density of the external force and $\uu_{0}$ is the initial velocity. In this paper, we assume $\mathbf{f}$ and $\uu_{0}$ are independent of $\e$ and are both in $L^{2}((0,T)\times \Omega;\R^{3})$. While, our main results still hold if  $\mathbf{f}$ and $\uu_{0}$ depend on $\e$ and converge strongly in $L^{2}((0,T)\times \Omega;\R^{3})$. The stress tensor $\eta_r(D\mathbf{u}_\varepsilon)$ is determined by the Carreau-Yasuda law:
$$\eta_r(D\mathbf{u}_\varepsilon)=(\eta_0-\eta_\infty)(1+\lambda| D\mathbf{u}_\varepsilon|^2)^{\frac{r}{2}-1}+\eta_\infty,\quad \eta_0\geq\eta_\infty, \ \lambda>0, \  r>1, $$
where $\eta_0$ is the zero-shear-rate viscosity, $\lambda$ is a time constant, $(r-1)$ is a dimensionless constant describing the slope in the {\em power law region} of log $\eta_r$
versus log $(|D(\uu_{\e})|)$.

The perforated domain $\Omega_{\e}$ under consideration is described as follows. Let $\Omega$ be a bounded domain of class $C^{2,\mu}, 0<\mu<1$. The holes in $\Omega$ are denoted by $T_{\varepsilon,k}$ which are assumed to satisfy
\begin{equation*}
 T_{\varepsilon,k}=\varepsilon x_k+\varepsilon T\subset\subset    \varepsilon Q_k,
\end{equation*}
where the cube $Q_k=(-\frac{1}{2},\frac{1}{2})^3+k$ and $x_k=x_0+k$ with $x_0\in(-\frac{1}{2},\frac{1}{2})^3,\; k\in {\mathbb{Z}}^3$; $T$ is a model hole which is assumed to be a closed domain  contained in $Q_0$ with $C^{2,\mu}$ boundary. The perforation parameter $\varepsilon$ is used to measure the mutual distance and size of the holes, and $\varepsilon x_k=\varepsilon x_0+\varepsilon k$ are the locations of the holes.

The perforated domain $\Omega_\varepsilon$ is then defined as:
\begin{equation}\label{1.2}
  \Omega_\varepsilon=\Omega\backslash\bigcup_{k\in K_\varepsilon}T_{\varepsilon,k},\quad {\rm where} \ K_\varepsilon=\{k\in {\mathbb{Z}}^3:\varepsilon\overline{Q}_{k}\subset\Omega\}.
\end{equation}

The study of homogenization problems in fluid mechanics have gained a lot interest. In particular, the homogenization of Stokes system in perforated domains has been systematically studied. In \cite{tartar}, Tartar considered the case where the size of the holes is proportional to the mutual distance of the holes and Darcy's law was derived.
Then Allaire \cite{ALLaire1,ALLaire2} considered general cases and showed that the homogenized equation are determined by the ratio $\sigma_\varepsilon$ between the size and the mutual distance of the holes:
$$\sigma_\varepsilon=\big(\frac{\varepsilon^d}{a_\varepsilon^{d-2}}\big)^\frac{1}{2},\quad d\geq3;\qquad\sigma_\varepsilon=\varepsilon\big|{\rm log}\frac{a_\varepsilon}{\varepsilon}\big|^\frac{1}{2},\quad d=2,$$
where $\varepsilon$ and $a_\varepsilon$ are used to measure the mutual distance of holes and the size of holes. Particularly, if $\displaystyle\lim_{\varepsilon\rightarrow 0}\sigma_\varepsilon=0$ corresponding to the case of large holes, the homogenized system is the Darcy's law; if $\displaystyle\lim_{\varepsilon\rightarrow 0}\sigma_\varepsilon=\infty$ corresponding to the case of small holes, there arise the same Stokes equations in homogeneous domains; if $\displaystyle\lim_{\varepsilon\rightarrow 0}\sigma_\varepsilon=\sigma_\ast\in (0,+\infty)$ corresponding to the case of critical size of holes, the homogenized equations are governed by the Brinkman's law--a combination of the Darcy's law and the original Stokes equations. Same results were shown in \cite{Yong} by employing a generalized cell problem inspired by Tartar \cite{tartar}.

Later, the homogenization study is extended to more complicated models describing fluid flows: Mikeli{\'c} \cite{Mik} for the nonstationary incompressible Navier-Stokes equations, Masmoudi \cite{masmoudi} for the compressible Navier-Stokes equations, Feireisl, Novotn{\'y} and Takahashi \cite{takahash} for the complete Navier-Stokes-Fourier equations. In all these studies, only the case where the size of holes is proportional to the mutual distance of the holes is considered and the Darcy's law is recovered in the limit.

Recently, cases with different sizes of holes are studied. Feireisl, Namlyeyeva and Ne{\v c}asov{\'a} \cite{FeNaNe} studied the case with critical size of holes for the incompressible Navier-Stokes equations and they derived Brinkman's law;  Yang and the first author \cite{luyong} studied the homogenization of evolutionary incompressible Navier-Stokes system with large and small size of holes. In \cite{ALL-NS1, ALL-NS6, ALLN}, with collaborators the first author considered the case of small holes for the compressible Navier-Stokes equations and it is shown that the homogenized equations remain the same as the original ones.  Oschmann and Pokorn{\' y} \cite{pokorn} also considered the case of small holes for the unsteady compressible Navier-Stokes equations for adiabatic exponent $\gamma> 3$ which improved the condition $\gamma> 6$ of \cite{ALLN}, and they showed that the homogenized equations keep unchanged. Bella and Oschmann \cite{oschmann} considered the homogenization of compressible Navier-Stokes equations for the case with randomly perforated domains with small size of holes and they get the same limiting equation. H{\"o}fer, Kowalczyk and Schwarzacher \cite{Hfer} studied the case of large holes for the compressible Navier-Stokes equations at low Mach number and derived the Darcy's law; Bella and Oschmann \cite{bella} also studied the case with critical size of holes for the compressible Navier-Stokes equations at low Mach number and they derived incompressible Navier-Stokes equations with Brinkman term. Bella, Feireisl and Oschmann \cite{feireisl}  considered the case of unsteady compressible Navier-Stokes equations at low Mach number under the assumption $\O_{\e}\to \O$ in  sense of Mosco's convergence and they derived the incompressible Navier-Stokes equations. Oschmann and Ne{\v c}asov{\'a} \cite{neasov} studied homogenization of the two-dimensional evolutionary compressible Navier-Stokes equations with very small holes and limiting equations remain unchanged.

 There are not many mathematical studies concerning the homogenization of non-Newtonian flows. Mikeli\'{c} and Bourgeat \cite{FL1} considered stationary case of Carreau-Yasuda type flows under the assumption $a_{\e}\sim \e$ and derived Darcy's law. Mikeli\'{c} summarized some theory of stationary non-Newtonian flows in Chapter 4 of \cite{hornung}. While for evolutionary non-Newtonian fluid equations, according to the authors' knowledge, there is no rigorous mathematical analysis results.  In this paper, we justify in porous media setting, the evolutionary Carreau-Yasuda model converges to Darcy's law.

\subsection{Notations and weak solutions}
We recall some notations of Sobolev spaces. Let $1\leq r\leq \infty$ and $\Omega$ be a bounded domain. We use the notation $L_0^r(\Omega)$ to denote the space of $L^r(\Omega)$ functions with zero mean value:
$$L_0^r(\Omega)=\Big\{f\in L^r(\Omega) \ : \ \int_\Omega f\, {\rm d}x=0\Big\}.$$
We use $W^{1,r}(\Omega)$ to denote classical Sobolev space, and $W^{1,r}_{0}(\Omega)$ denotes the completion of $C_c^\infty(\Omega)$ in $W^{1,r}(\Omega)$. Here $C_c^\infty(\Omega)$ is the space of smooth functions compactly supported in $\Omega$. We use $W^{-1,r'}(\Omega)$ to denote the dual space of $W^{1,r}_{0}(\Omega)$. For $1\leq r<\infty$, $W^{1,r}(\mathbb{R}^3)=W^{1,r}_0(\mathbb{R}^3)$. We introduce the functional space $W_{0,\rm div}^{1,r}(\Omega),  \ 1\leq r \leq \infty$ by
$$W_{0,\rm div}^{1,r}(\Omega)=\left\{u\in W_0^{1,r}(\Omega;\mathbb{R}^3): \ {\rm div}\,u=0 \ {\rm in} \, \Omega\right\}.$$

Now we introduce the definition of finite energy weak solutions to \eqref{1.3}:
\begin{defi}\label{def-weak}
Let $T>0$. We say that $\mathbf{u}_{\varepsilon}$ is a finite energy weak solution of \eqref{1.3} in $(0,T)\times \Omega_{\e}$ provided
\begin{itemize}
\item $\mathbf{u}_{\varepsilon}\in C_{\rm weak}([0,T); L^2(\Omega_\varepsilon;\R ^3))\cap L^2(0,T; W_{0, \rm div}^{1,2}(\Omega_\varepsilon))\cap L^r(0,T; W_{0, \rm div}^{1,r}(\Omega_\varepsilon))$.

\item The integral identity
\ba
&\int_0^T\int_{\Omega_{\varepsilon}}-\varepsilon^2\mathbf{u}_{\varepsilon}\cdot \partial_t\varphi-\mathbf{u}_{\e}\otimes \mathbf{u}_{\e}:\nabla\varphi+\eta_r(D\mathbf{u}_\varepsilon)D\mathbf{u}_{\varepsilon}:D\varphi\, {\rm d}x{\rm d}t\\
&=\int_0^T\int_{\Omega_\varepsilon}\mathbf{f}\cdot\varphi \, {\rm d}x{\rm d}t+\varepsilon^2\int_{\Omega_\varepsilon}\mathbf{u}_0\cdot\varphi(0,\cdot)\, {\rm d}x
\nn\ea
holds for any test function\;$\varphi\in C_c^\infty([0,T)\times\Omega_\varepsilon;\mathbb{R}^3),\ {\rm div}_x\varphi=0$.

\item The energy inequality
\ba\label{energy-ineq}
\int_{\Omega_\varepsilon}\frac{\varepsilon^2}{2}\mathbf{u}_{\varepsilon}^2\,{\rm d}x+\int_0^t\int_{\Omega_\varepsilon}\eta_r(D\mathbf{u}_\varepsilon)| D\mathbf{u}_{\varepsilon}|^2\,{\rm d}x{\rm d}t\leq\int_{\Omega_\varepsilon}\frac{\varepsilon^2}{2}\mathbf{u}_0^2\,{\rm d}x+\int_0^t\int_{\Omega_\varepsilon}\mathbf{f}\cdot \mathbf{u}_{\varepsilon}\,{\rm d}x{\rm d}t
\ea
holds for a.a. $t\in(0,T)$.

\end{itemize}

\end{defi}

The classical theory from Ladyzhenskaya \cite{lad}, Theorem 1.1 in \cite{malek} and Theorem 1.3 in \cite{wolf} gives the existence of at least one weak solution $\mathbf{u}_\varepsilon\in C_{\rm weak}([0,T); L^2(\Omega_\varepsilon;\R ^3))\cap L^r(0,T; W_{0, \rm div}^{1,r}(\Omega_\varepsilon))$ for $r> 2$ and $\mathbf{u}_\varepsilon\in C_{\rm weak}([0,T); L^2(\Omega_\varepsilon;\R ^3))\cap L^2(0,T; W_{0, \rm div}^{1,2}(\Omega_\varepsilon))$ for $1<r\leq2$.

\medskip

For brevity we use $C$ to denote a constant independent of $\varepsilon$ throughout the paper, while the value of $C$ may differ from line to line.

\subsection{Restriction, extension, and some useful lemmas}
Our goal is to obtain the limit system in homogeneous domains without holes, so we need to extend $(\mathbf{u}_\varepsilon,p_\varepsilon)$ to the whole of $\Omega$. Due to the zero boundary conditions on $\uu_{\e}$, it is nature to extend $\mathbf{u}_\varepsilon$ by zero to the holes.  However, the extension of the pressure is more delicate. It is defined by the restriction operator due to Tartar \cite{tartar} for the case where the size of the holes is proportional to their mutual distance and extended to general sizes of holes by Allaire \cite{ALLaire1,ALLaire2}. For $\Omega_\varepsilon$ defined in (\ref{1.2}), there exists a linear operator, named restriction operator, $R_\varepsilon:W_0^{1,q}(\Omega;{\mathbb{R}}^3)\rightarrow W_0^{1,q}(\Omega_\varepsilon;{\mathbb{R}}^3)$ $(1<q<\infty)$ such that:
\ba\label{def-restriction}
&u\in W_0^{1,q}(\Omega_\varepsilon;{\mathbb{R}}^3)\Longrightarrow R_\varepsilon(\tilde{u})=u\ {\rm in}\ \Omega_\varepsilon, \ {\rm where} \ \tilde{u} \ {\rm is \ the \ zero \ extension \ of} \ u,\\
&{\rm div}\,u=0 {\rm \ in} \ \Omega\Longrightarrow {\rm div}\,R_\varepsilon (u)=0 \ \rm in \ \Omega_\varepsilon,\\
 & \|\nabla R_\varepsilon(u)\|_{L^q(\Omega_\varepsilon)}\leq C(\varepsilon^{-1}\|u\|_{L^q(\Omega)}+\|\nabla u\|_{L^q(\Omega)}).
\ea
The construction of such a restriction operator can be found in \cite{Mik}.  Later Allaire \cite{ALLaire1,ALLaire2} constructed such type of restriction operators for general sizes of holes in $L^{2}$ framework. Recently, following the construction of Allaire, the first author \cite{Y. Lu} gave a construction of restriction operators  for general sizes of holes in general $L^{q}$ framework.

The extension $\tilde{p}_\varepsilon$ of the pressure $p_\varepsilon$ with $\nabla p_{\e} \in W^{-1,q'}(\Omega_{\e}; \R^{3})$ is then defined through the following dual formulation:
\begin{equation*}
  \langle\nabla\tilde{p}_\varepsilon,\varphi\rangle_{\Omega}=\langle\nabla p_\varepsilon,R_\varepsilon(\varphi)\rangle_{\Omega_\varepsilon},\qquad \forall \,\varphi\in W^{1,q}_{0}(\Omega;{\mathbb{R}}^3).
\end{equation*}
Such an extension $\tilde p_{\e}$ is well defined due to the three properties in \eqref{def-restriction} of the restriction operator.

 \medskip
Now we introduce several useful conclusions which will be frequently used throughout this paper.
Let's first introduce Poincar\'e inequality in porous media. One may find this proof in \cite{tartar,Mik}.
\begin{lemma}\label{lem-Poincare}
For each $u\in W_0^{1,q}(\Omega_\varepsilon;\mathbb{R}^3),\;1<q<\infty$, where $\Omega_{\varepsilon}$ is defined in (\ref{1.2}). Then there holds
\begin{equation}\label{1.7}
 \|u\|_{L^q(\Omega_\varepsilon)}\leq C\varepsilon\|\nabla u\|_{L^q(\Omega_\varepsilon)}.
\end{equation}

\end{lemma}

\medskip

Next we introduce the restriction operator which concerns functions with time variable. For each $\varphi\in L^{p}(0,T;W_{0}^{1,q}(\Omega))$ with $1\leq p\leq\infty,\,1<q<\infty$, the restriction $R_\varepsilon(\varphi)$ is taken only on spatial variable:
\begin{equation}\label{res time def}
  R_\varepsilon(\varphi)(\cdot,t)=R_\varepsilon\big(\varphi(\cdot,t)\big)(\cdot),\quad \mbox{for\, each} \ t\in(0,T).
\end{equation}
Then we can get the same properties as in (\ref{def-restriction}) for each $t\in (0,T)$. Moreover, it is rather straightforward to deduce from \eqref{def-restriction} and \eqref{1.7} the following lemma:
\begin{lemma}
Let $\Omega$ be a bounded domain of class $C^1$ and $\Omega_{\e}$ be defined in \eqref{1.2}. Let $\varphi\in L^{p}(0,T;W_{0}^{1,q}(\Omega;\mathbb R^3)),\ 1\leq p\leq\infty,\;1<q<\infty$. Then we have
\begin{equation}\label{1.5}
\| R_\varepsilon (\varphi)\|_{L^{p}(0,T;L^{q}(\Omega_\varepsilon))}+\e \|\nabla R_\varepsilon (\varphi)\|_{L^{p}(0,T;L^{q}(\Omega_\varepsilon))}\leq C(\|\varphi\|_{L^{p}(0,T;L^{q}(\Omega))}+\e\|\nabla \varphi\|_{L^{p}(0,T;L^{q}(\Omega))}).
\end{equation}

\end{lemma}

\medskip

We finally give the following Korn type inequality, see for example Chapter 10 in \cite{9}:
\begin{lemma} \label{lem-Korn}(Korn inequality) Let $\Omega$ be a bounded domain of class $C^1$ and $\Omega_{\e}$ be defined in \eqref{1.2}.  Let $1<q<\infty$. For arbitrary $u\in W_0^{1,q}(\Omega_{\e};\mathbb{R}^3)$, there holds
\begin{equation}\label{1.8}
  \|\nabla u\|_{L^q(\Omega_{\e})}  \leq C(q)\|Du\|_{L^q(\Omega_{\e})},
\end{equation}
where $C(q)$ is independent of $\e$.
\end{lemma}

\subsection{Main results}
We now state our homogenization results, where the limits are taken up to possible extractions of subsequences. We shall follow the idea of  Mikeli\'{c} \cite{Mik} and Teman \cite{teman} to consider new equations by integrating original equations with respect to time variable.  Let $(\mathbf{u}_\varepsilon,p_\varepsilon)$ be a finite energy weak solution of equations (\ref{1.3}). Introduce
\begin{equation}\label{1.10.5}
  U_\varepsilon = \int_0^t\mathbf{u}_\varepsilon\, {\rm d}s, \;G_\varepsilon=\int_0^t(\mathbf{u}_\varepsilon\cdot\nabla)\mathbf{u}_\varepsilon\, {\rm d}s, \;H_\varepsilon=\int_0^t(1+\lambda| D\mathbf{u}_\varepsilon|^2)^{\frac{r}{2}-1}D\mathbf{u}_\varepsilon \, {\rm d}s, \;F=\int_0^t\mathbf{f} \, {\rm d}s.
\end{equation}
Then we have $U_\varepsilon\in C([0,T];\;W^{1,2}_{0, \rm div} (\Omega_\varepsilon)),\;G_\varepsilon\in C([0,T];L^{\frac{3}{2}}(\Omega_\varepsilon)),\;F\in C([0,T];L^2(\Omega_\varepsilon))$ and
\begin{equation*}
H_\varepsilon\in
\begin{cases}
C([0,T];L^{2}(\Omega_\varepsilon)) & 1<r\leq2,\\
C([0,T];L^{\frac{r}{r-1}}(\Omega_\varepsilon))& r>2.
\end{cases}
\end{equation*}

The classical theory of Stokes equations ensure the existence of
\begin{equation*}
P_\varepsilon\in
\begin{cases}
C_{\rm weak}([0,T]; L^2(\Omega_\varepsilon)) & 1<r\leq2,\\
C_{\rm weak}([0,T]; L^{\frac{r}{r-1}}(\Omega_\varepsilon))& r>2,
\end{cases}
\end{equation*}
such that for each $t\in [0,T]$,
\begin{equation}\label{1.11.1}
  \nabla P_\varepsilon=F-\varepsilon^2(\mathbf{u}_\varepsilon-\mathbf{u}_0)+\frac{\eta_\infty}{2}\Delta U_\varepsilon-G_\varepsilon+(\eta_0-\eta_\infty){\rm div}\,H_\varepsilon \quad \mbox{in}  \ \mathcal{D}'(\Omega_{\e}).
\end{equation}

Now we are ready to state the main theorem:
\begin{theorem}\label{thm-1}
Let $1<r<\infty$. Let $(\mathbf{u}_\varepsilon,p_\varepsilon)$ be a finite energy weak solution of equations (\ref{1.3}), and $\tilde{\mathbf{u}}_\varepsilon$ is the zero extension of $\mathbf{u}_\varepsilon$. The extension $\tilde{P}_\varepsilon$ is defined through the restriction operator as follows,
$$\langle\nabla\tilde{P}_\varepsilon,\varphi\rangle_{(0,T)\times\Omega}=\langle\nabla P_\varepsilon,R_{\varepsilon}(\varphi)\rangle_{(0,T)\times\Omega_{\varepsilon}},\quad for \ all \ \varphi\in C_c^\infty((0,T)\times\Omega),$$
where $P_\varepsilon$ is defined in (\ref{1.11.1}). Let $\tilde{p}_\e=\partial_t \tilde{P}_\varepsilon$ be the extension of $p_\e$. Then we can find $\mathbf{u}\in L^2((0,T)\times\Omega)$ and
\begin{equation*}
p\in
\begin{cases}
W^{-1,2}(0,T;L^2(\Omega)) & 1<r\leq2,\\
W^{-1,\frac{r}{r-1}}(0,T;L^{\frac{r}{r-1}}(\Omega))& r>2,
\end{cases}
\end{equation*}
which satisfy
$$ \varepsilon^{-2}\tilde{\mathbf{u}}_\varepsilon\rightarrow \mathbf{u} \  weakly \ in \ L^2((0,T)\times\Omega),$$
\begin{equation*}
\tilde{p}_\varepsilon\rightarrow p\ weakly  \ in
\begin{cases}
W^{-1,2}(0,T;L^2(\Omega)) & 1<r\leq2,\\
W^{-1,\frac{r}{r-1}}(0,T;L^{\frac{r}{r-1}}(\Omega))& r>2.
\end{cases}
\end{equation*}

Moreover, the limit $(\mathbf{u},p)$ satisfies the Darcy's law:
\begin{equation}\label{1.10.1}
  \frac{1}{2}\eta_0\mathbf{u}=A(\mathbf{f}-\nabla p) \quad \mbox{in} \ \mathcal{D'}((0,T)\times \Omega).
\end{equation}
\end{theorem}

\medskip

We give several remarks concerning our main results and main ideas of proof:

\begin{remark}

\begin{itemize}

\item The permeability tensor $A$ which appears in (\ref{1.10.1}) is a constant positive definite matrix defined in (\ref{1.21.0}).

\item Mikeli\'{c} considered the homogenization of nonstationary Navier-Stokes equations in \cite{Mik}, namely for the case $r=2$, and derived Darcy's law. Observing the strong convergence of the nonlinear viscosity coefficient $\eta_r(D\mathbf{u}_\varepsilon)$ to constant $\eta_0$, we derived the Darcy's law for arbitrary $r>1$.

\item The main difficulty compared to the steady case considered in \cite{FL1} lies in dealing with the estimates of the pressure $p_{\e}$, which lies in some negative order Sobolev space with respect to time variable due to the presence of the time derivative term $\d_{t} \uu_{\e}$.  We shall follow the idea of Mikeli\'{c} \cite{Mik} by integrating original equations with respect to time variable, and this allows us to define the extension of the pressure pointwisely in $t$.

\end{itemize}

\end{remark}

The rest of the paper is devoted to the proof of Theorem \ref{thm-1}. In Section 2, we derive the uniform estimates of the velocity field and pressure. In Section 3, we employ the cell problem to modify test functions, and then pass to limit in the weak formulation of the new equations to derive the limit system---Darcy's law.

\section{Uniform estimates}
In this section, we derive the uniform estimates of the solutions.  The estimates of the velocity follows from the energy inequality by using the Poincar\'e inequality and  the Korn inequality (see Lemma \ref{lem-Poincare} and \ref{lem-Korn}). Concerning the pressure $p_{\e}$,  we will not deduce the estimates of $p_{\e}$ directly. Instead, we will consider a proper extension of  its time integration $P_{\e}$ given in (\ref{1.11.1}). Such an extension is defined by a dual formula using restriction operator pointwisely in $t$.  The estimates of the extension of ${P_\varepsilon}$ follow from the estimates of $U_{\e}$ and the estimates of the restriction operator.

\subsection{Estimates of velocity field}

Based on the energy inequality \eqref{energy-ineq}, we can derive the following estimates of velocity field $\uu_{\e}$:
\begin{proposition}\label{pro2.1}
Let $\mathbf{u}_\varepsilon$ be a weak solution of (\ref{1.3}) in the sense of Definition \ref{def-weak}. There holds
\ba\label{2.1.01}
&\|\nabla \mathbf{u}_\varepsilon\|_{L^2((0,T)\times\Omega_\varepsilon)} \leq  C\varepsilon, \quad &&\| \mathbf{u}_\varepsilon\|_{L^2((0,T)\times\Omega_\varepsilon)} \leq  C\varepsilon^{2}, \quad \| \mathbf{u}_\varepsilon\|_{L^\infty(0,T;L^{2}(\Omega_\varepsilon))} \leq  C,\\
&\|\nabla \mathbf{u}_\varepsilon\|_{L^r((0,T)\times\Omega_\varepsilon)} \leq  C\varepsilon^{\frac 2r}, \quad &&\| \mathbf{u}_\varepsilon\|_{L^r((0,T)\times\Omega_\varepsilon)} \leq  C\varepsilon^{\frac 2 r +1}, \quad \mbox{if} \ r>2.
\ea

\end{proposition}

\begin{proof}

We first rewrite the energy inequality \eqref{energy-ineq} as
\ba\label{est-u-1}
  &\frac{\varepsilon^2}{2}\int_{\Omega_\varepsilon}\mathbf{u}_\varepsilon^2\,{\rm d}x+\int_0^t\int_{\Omega_\varepsilon}\eta_\infty| D\mathbf{u}_\varepsilon|^2+(\eta_0-\eta_\infty)(1+\lambda|D\mathbf{u}_\varepsilon|^2)^{\frac{r}{2}-1}| D\mathbf{u}_\varepsilon|^2\,{\rm d}x{\rm d}t\\
  &\quad \leq\int_0^t\int_{\Omega_\varepsilon}\mathbf{f}\cdot \mathbf{u}_\varepsilon \,{\rm d}x{\rm d}t
  +\frac{\varepsilon^2}{2}\int_{\Omega_\varepsilon}\mathbf{u}_0^2\,{\rm d}x, \quad
 \mbox{for a.a. $0<t\leq T$.}
\ea

Applying the Poincar\'e inequality in porous media and the Korn inequality (see Lemma \ref{lem-Poincare} and \ref{lem-Korn}) gives
\ba\label{est-u-2}
\int_0^T\int_{\Omega_\varepsilon}\mathbf{f}\cdot \mathbf{u}_\varepsilon \,{\rm d}x{\rm d}t
\leq C \e  \int_0^T \|\mathbf{f}\|_{L^{2}(\Omega)} \|D\mathbf{u}_\varepsilon\|_{L^{2}(\Omega_{\e})}  {\rm d}t.
\ea
Then from \eqref{est-u-1} and \eqref{est-u-2}, and the assumptions that $\ff$ and $\uu_{0}$ are independent of $\e$ and are in $L^{2}((0,T)\times \Omega)$,   we deduce
\ba
\|D\mathbf{u}_\varepsilon\|^2_{L^2((0,T)\times\Omega_\varepsilon)} \leq C\varepsilon \|D\mathbf{u}_\varepsilon\|_{L^2((0,T)\times\Omega_\varepsilon)} \| \ff \|_{L^2((0,T)\times\Omega_\varepsilon)}  + C\varepsilon^2.
\nn\ea
Again by the Poincar\'e inequality (\ref{1.7}) and the Korn inequality (\ref{1.8}), we obtain the $L^{2}$ estimates in $\eqref{2.1.01}_{1}$.

\medskip

If $r>2$,  using the estimates in $\eqref{2.1.01}_{1}$, we deduce from \eqref{est-u-1} that
\ba
\|D\mathbf{u}_\varepsilon\|^r_{L^r((0,T)\times\Omega_\varepsilon)} \leq C\varepsilon^2,
\nn\ea
and the $L^{r}$ estimates in $\eqref{2.1.01}_{2}$ follow from the Poincar\'e inequality and the Korn inequality.

\end{proof}

By the uniform estimates of velocity in Proposition \ref{pro2.1}, we have the following uniform estimates:
\begin{corollary}\label{cor-est-U}
Let $U_{\e}$, $G_{\e}$ and $H_{\e}$ be defined as in \eqref{1.10.5}.  Then
\ba\label{est-U}
& \|U_\varepsilon\|_{W^{1,2}(0,T; W^{1,2}_{0}(\Omega_\varepsilon))}\leq C \e, \quad \|U_\varepsilon\|_{W^{1,2}(0,T; L^{2}(\Omega_\varepsilon))}\leq C \e^{2}, \\
& \|U_\varepsilon\|_{W^{1,r}(0,T; W^{1,r}_{0}(\Omega_\varepsilon))}\leq C \e^{\frac 2r}, \quad \|U_\varepsilon\|_{W^{1,r}(0,T; L^{r}(\Omega_\varepsilon))}\leq C \e^{\frac 2 r +1}, \quad \mbox{if} \  r >2,\\
& \|G_\varepsilon\|_{W^{1,1}(0,T; L^{\frac 32}(\Omega_\varepsilon))}\leq C \e^{2}, \\
& \|H_\varepsilon\|_{W^{1,2}(0,T; L^{2}(\Omega_\varepsilon))}\leq C \e,  \quad \mbox{if} \ 1< r \leq 2,\\
& \|H_\varepsilon\|_{W^{1,\frac{r}{r-1}}(0,T; L^{\frac{r}{r-1}}(\Omega_\varepsilon))} \leq C \e,  \quad  \mbox{if} \ r>2.
\ea
\end{corollary}

\begin{proof}

The estimates for $U_{\e}$ in $\eqref{est-U}_{1}$ and $\eqref{est-U}_{2}$ follow immediately from its definition and the uniform estimates of ${\mathbf u}_{\e}$ in \eqref{2.1.01}.

\medskip

Using Sobolev embedding and H\"older's inequality gives
\ba
\| \uu_{\e} \cdot \nabla \uu_{\e}\|_{L^{\frac{3}{2}}(\Omega_{\e})} \leq  \| \uu_{\e} \|_{L^{6}(\Omega_{\e})}\| \nabla \uu_{\e} \|_{L^{2}(\Omega_{\e})} \leq C\| \nabla \uu_{\e} \|_{L^{2}(\Omega_{\e})}^{2}.
\nn\ea
Thus
$$
\| \uu_{\e} \cdot \nabla \uu_{\e}\|_{L^{1}(0,T; L^{\frac{3}{2}}(\Omega_{\e}))} \leq C\| \nabla \uu_{\e} \|_{L^{2}((0,T)\times\Omega_{\e})}^{2} \leq C \e^{2},
$$
which gives the estimates of $G_{\e}$ in \eqref{est-U}.

\medskip

We turn to the estimates of $H_{\e}$.  If $1<r\leq2$, there holds
\ba\label{2.4}
|(1+\lambda| D\mathbf{u}_\varepsilon|^2)^{\frac{r}{2}-1}D\mathbf{u}_\varepsilon | \leq | D \uu_{\e}|.
\nn \ea
Therefore
$$\|H_\varepsilon\|_{W^{1,2}(0,T; L^{2}(\Omega_\varepsilon))}\leq  C\|(1+\lambda| D\mathbf{u}_\varepsilon|^2)^{\frac{r}{2}-1}D\mathbf{u}_\varepsilon \|_{L^{2}((0,T)\times \Omega_\varepsilon)}   \leq C \| D \uu_{\e}\|_{L^{2}((0,T)\times \Omega_\varepsilon)}   \leq  C\e.
$$

If $r>2$, we have
\ba\label{2.5}
|(1+\lambda| D\mathbf{u}_\varepsilon|^2)^{\frac{r}{2}-1}D\mathbf{u}_\varepsilon | \leq C  | D \uu_{\e}|  + C | D \uu_{\e}|^{r-1}.
\nn \ea
Therefore
\ba
 \|H_\varepsilon\|_{W^{1,\frac{r}{r-1}}(0,T; L^{\frac{r}{r-1}}(\Omega_\varepsilon))} &  \leq  C \|  (1+\lambda| D\mathbf{u}_\varepsilon|^2)^{\frac{r}{2}-1}D\mathbf{u}_\varepsilon \|_{L^{\frac{r}{r-1}}((0,T)\times \Omega_\varepsilon)} \\
 & \leq  C \| D \uu_{\e}\|_{L^{\frac{r}{r-1}}((0,T)\times\Omega_\varepsilon)}  + C \| D \uu_{\e}\|_{L^{r}((0,T)\times \Omega_\varepsilon)}^{r-1}  \\
  & \leq  C \| D \uu_{\e}\|_{L^{2}((0,T)\times\Omega_\varepsilon)}  + C \| D \uu_{\e}\|_{L^{r}((0,T)\times \Omega_\varepsilon)}^{r-1}\\
    & \leq C  \e + C \e^{\frac{2(r-1)}{r}}\leq C \e,
\nn \ea
where we use the fact $\frac{r}{r-1} < 2$ and $\frac{2(r-1)}{r} >1$ if $r>2$.

\end{proof}

\subsection{Estimates of pressure: extension}
Next we will define the extension of pressure and derive corresponding estimates based on the uniform estimates in (\ref{est-U}).
\begin{proposition}\label{pro2.1.1}
Let $\tilde{P_\varepsilon}$ be the extension of $P_\varepsilon$ defined by the restriction operator as follows
\begin{equation}\label{pre def}
 \langle\nabla \tilde{P_\varepsilon},\varphi\rangle_{(0,T)\times\Omega}=\langle\nabla P_\varepsilon,R_\varepsilon(\varphi)\rangle_{(0,T)\times\Omega_\varepsilon},\quad \forall\varphi\in C_c^\infty((0,T)\times\Omega),
\end{equation}
where $P_\varepsilon$ is defined in (\ref{1.11.1}) and the restriction operator $R_{\e}$ is given in (\ref{res time def}) satisfying the estimates (\ref{1.5}).
Then we have the following estimates:
\begin{equation}\label{pressure est}
\big|\langle\nabla \tilde{P_\varepsilon},\varphi\rangle_{(0,T)\times\Omega}\big|\leq C
\begin{cases}
\|\varphi\|_{L^{2}((0,T)\times\Omega)}+\e\| \nabla \varphi\|_{L^{2}((0,T)\times\Omega)}\quad 1<r\leq2,\\
\|\varphi\|_{L^{r}((0,T)\times\Omega)}+\e\| \nabla \varphi\|_{L^{r}((0,T)\times\Omega)}\quad r>2.
\end{cases}
\end{equation}
\end{proposition}
\begin{proof}

By $(\ref{1.11.1})$, we have
\begin{align*}
\langle\nabla \tilde{P_\varepsilon},\varphi\rangle_{(0,T)\times\Omega}&=\langle\nabla P_\varepsilon,R_\varepsilon(\varphi)\rangle_{(0,T)\times\Omega_\varepsilon} \\
   &=\langle F-\varepsilon^2 (\mathbf{u}_\varepsilon-\mathbf{u}_0)+\frac{\eta_\infty}{2}\Delta U_\varepsilon-G_\varepsilon
  +(\eta_0-\eta_\infty){\rm div }\,H_\varepsilon, R_\varepsilon(\varphi)\rangle_{(0,T)\times\Omega_\varepsilon}.
\end{align*}
Using the fact $F$ and $\uu_{0}$ are both in $L^{2}((0,T)\times \Omega;\R^{3})$ and $\eqref{2.1.01}_{1}$ implies
\begin{align*}
  &\big|\langle F-\varepsilon^2 (\mathbf{u}_\varepsilon-\mathbf{u}_0), R_\varepsilon(\varphi)\rangle_{(0,T)\times\Omega_\varepsilon}\big|\leq C \|R_\varepsilon(\varphi)\|_{L^2((0,T)\times\Omega_\varepsilon)}.
\end{align*}
By the estimates in (\ref{est-U}) and Sobolev embedding inequality, we have:
\ba
\big|\langle \Delta U_\varepsilon, R_\varepsilon(\varphi) \rangle_{(0,T)\times\Omega_\varepsilon}\big| &\leq \|\nabla U_\varepsilon\|_{L^2((0,T)\times\Omega_\varepsilon)}\|\nabla R_\varepsilon(\varphi)\|_{L^2((0,T)\times\Omega_\varepsilon)}\leq C\e \|\nabla R_\varepsilon(\varphi)\|_{L^2((0,T)\times\Omega_\varepsilon)},\\
  \big|\langle G_\varepsilon, R_\varepsilon(\varphi) \rangle_{(0,T)\times\Omega_\varepsilon}\big| &\leq\|G_\e\|_{L^2(0,T;L^{\frac{3}{2}}(\O_\e))}\|
R_\varepsilon(\varphi)\|_{L^2(0,T;L^3(\Omega_\varepsilon))}\leq C\e^2 \|\nabla R_\varepsilon(\varphi)\|_{L^2((0,T)\times\Omega_\varepsilon)},\\
  \big|\langle \dive H_\varepsilon, R_\varepsilon(\varphi) \rangle_{(0,T)\times\Omega_\varepsilon}\big| & \leq\|H_\e\|_{L^2((0,T)\times\Omega_\varepsilon)}\| \nabla R_\varepsilon(\varphi)\|_{L^2((0,T)\times\Omega_\varepsilon)} \\
  & \leq C\e\| \nabla R_\varepsilon(\varphi)\|_{L^2((0,T)\times\Omega_\varepsilon)}\quad \mbox{if} \ 1< r \leq 2,\\
  \big|\langle \dive H_\varepsilon, R_\varepsilon(\varphi) \rangle_{(0,T)\times\Omega_\varepsilon}\big| & \leq\|H_\e\|_{L^{\frac{r}{r-1}}((0,T)\times\Omega_\varepsilon)}\| \nabla R_\varepsilon(\varphi)\|_{L^r((0,T)\times\Omega_\varepsilon)} \\
  & \leq C\e\| \nabla R_\varepsilon(\varphi)\|_{L^r((0,T)\times\Omega_\varepsilon)}\quad \mbox{if} \ r> 2.
\nn\ea
Hence, by the estimates of restriction operator in (\ref{1.5}), we have  for $1<r\leq2$,
\begin{align*}
\big|\langle\nabla \tilde{P_\varepsilon},\varphi\rangle_{(0,T)\times\Omega}\big|&\leq C\big( \| R_\varepsilon(\varphi)\|_{L^2((0,T)\times\Omega_\varepsilon) }+  \e\| \nabla R_\varepsilon(\varphi)\|_{L^2((0,T)\times\Omega_\varepsilon)}\big)\\
&\leq C\big(\|\varphi\|_{L^2((0,T)\times\Omega)}+\e\| \nabla \varphi\|_{L^2((0,T)\times\Omega)}\big).
\end{align*}
For $r>2$,
\begin{align*}
\big|\langle\nabla \tilde{P_\varepsilon},\varphi\rangle_{(0,T)\times\Omega}\big|&\leq C\big( \| R_\varepsilon(\varphi)\|_{L^r((0,T)\times\Omega_\varepsilon) }+\e\| \nabla R_\varepsilon(\varphi)\|_{L^r((0,T)\times\Omega_\varepsilon)} \big) \\
&\leq C\big(\|\varphi\|_{L^r((0,T)\times\Omega)}+\e\| \nabla \varphi\|_{L^r((0,T)\times\Omega)}\big).
\end{align*}

The proof is thus completed.
\end{proof}

\section{Homogenization process}

This section is devoted to the limit passage and derive the limit equations. We first introduce the cell problem which is used to modify test functions. Then by the estimates in Proposition \ref{pro2.1} and Corollary \ref{cor-est-U}, we can pass to the limit term by term to get the limit system.

\subsection{Cell Problem}
 To obtain the limit system, a natural way is to pass $\varepsilon\rightarrow0$ in the weak formulation of (\ref{1.3}), a proper surgery on $C_c^\infty(\Omega)$ test functions needs to be done such that the test functions vanish on the holes and then become good test functions for the original equations in $\Omega_\varepsilon$. To this issue, Tartar \cite{tartar} considered the Stokes equations where the size of the holes is proportional to the mutual distance of the holes. Then near each single hole in $\varepsilon Q_k$ in the perforated domain $\Omega_\varepsilon$, after a scaling of size $\varepsilon^{-1}$, there arises typically the following problem, named cell problem:

Let $(w^i,\pi^i)(i=1,2,3)$ be a $Q_0$-periodic solution of the following cell problem
\begin{equation}\label{2.9.5}
\begin{cases}
-\Delta w^i+\nabla\pi^i=e^i&{\rm in} \ Q_0\setminus T,\\
{\rm div}\, w^i=0&{\rm in} \ Q_0\setminus T,\\
w^i=0& {\rm on}\ T.
\end{cases}
\end{equation}
Here $\{e^i\}_{i=1,2,3}$ is the standard Euclidean coordinate of $\mathbb R^3$. The cell problem (\ref{2.9.5}) admits a unique weak solution $(w^i,\pi^i)\in W^{1,2}(Q_0\setminus T;\mathbb{R}^3)\times L_0^2(Q_0\setminus T)$ with $(w^i, \pi^i)$ $Q_0-{\rm periodic}$. Moreover, under the assumption $T$ is of class $C^{2, \mu}$, one has
\ba\label{est-wi}
\|w^{i}\|_{W^{1, \infty}(Q_{0}\setminus T)} + \|\pi^{i}\|_{L^{\infty}(Q_{0}\setminus T)}  \leq C.
\ea
The permeability tensor $A$ is defined as
\begin{equation}\label{1.21.0}
A_{i,j}=\int_{Q_0}w^i_j(y)\,{\rm d}y,\qquad A=(A_{i,j})_{1\leq i,j\leq3},
\end{equation}
where $w^i_j$ denotes the  $j$-th component of vector $w^i$. It is shown in \cite{tartar} that $A$ is symmetric and positive definite.

Next we set
\begin{equation*}
  w^{i,\varepsilon}(x)=w^i(\frac{x}{\varepsilon}),\qquad \pi^{i,\varepsilon}(x)=\pi^i(\frac{x}{\varepsilon}).
\end{equation*}
Then $(w^{i,\varepsilon}(x), \pi^{i,\varepsilon}(x))$ satisfies the following equations
\begin{equation}\label{2.12.5}
\begin{cases}
-\varepsilon\nabla\pi^{i,\varepsilon}+\varepsilon^2\Delta w^{i,\varepsilon}+e^i=0&{\rm in}\;\varepsilon Q_0\setminus\varepsilon T,\\
{\rm div}\;w^{i,\varepsilon}=0&{\rm in}\;\varepsilon Q_0\setminus\varepsilon T,\\
w^{i,\varepsilon}=0&{\rm on}\;\varepsilon T,\\
(w^{i,\varepsilon},\pi^{i,\varepsilon}) \ {\rm is} \ \varepsilon Q_0-{\rm periodic}.
\end{cases}
\end{equation}
{Moreover, it follows from \eqref{est-wi} that}
\begin{equation}\label{2.13}
  \|w^{i,\varepsilon}\|_{L^\infty(\Omega_\varepsilon)}\leq C, \quad  \|\nabla w^{i,\varepsilon}\|_{L^\infty(\Omega_\varepsilon)}\leq C\varepsilon^{-1},\quad
   \|\pi^{i,\varepsilon}\|_{L^\infty(\Omega_\varepsilon)}\leq C.
\end{equation}

By the fact that $w^i$ is $ Q_0-{\rm periodic}$, using (\ref{2.13}) implies
\begin{equation}\label{w^i con}
  w^{i,\varepsilon}\rightarrow \bar{w}^i := \int_{Q_{0}} w^{i} (y)\, {\rm d}y\quad {\rm weakly \ in \ } L^r(\O),
\end{equation}
for each $1<r<\infty$.

\subsection{Passing to the limit}

Our main theorem actually follows from the following key proposition:
\begin{proposition}\label{pro2.2}
Let $(U_{\e},P_{\e})$ be the solutions of the equation (\ref{1.11.1}). Let $\tilde{P}_\varepsilon$ be the extension of $P_{\e}$ defined in (\ref{pre def}) and $\tilde{U}_\varepsilon$ be the zero extension of $U_{\e}$. Then we can find $U\in L^2((0,T)\times\Omega)$ and
\begin{equation}\label{p def}
P\in
\begin{cases}
L^2((0,T)\times\Omega) & 1<r\leq2,\\
L^{\frac{r}{r-1}}((0,T)\times\Omega) & r>2,
\end{cases}
\end{equation}
such that
\begin{equation}\label{vel conver}
  \varepsilon^{-2}\tilde{U}_\varepsilon\rightarrow U \  weakly \ in \ L^2((0,T)\times\Omega),
\end{equation}
\begin{equation}\label{p con}
\tilde{P_\varepsilon}\rightarrow P \  weakly \  in
\begin{cases}
L^2((0,T)\times\Omega) & 1<r\leq2,\\
L^{\frac{r}{r-1}}((0,T)\times\Omega) & r>2.
\end{cases}
\end{equation}
Moreover, the limit $(U,P)$ satisfies the Darcy's law:
\ba\label{Darcy-U}\frac{1}{2}\eta_0U=A(F-\nabla P)\quad {\rm in}\,\mathcal{D'}((0,T)\times\Omega).\ea
Here the permeability tensor $A$ is a constant positive definite matrix determined in (\ref{1.21.0}).
\begin{proof}
The convergence in (\ref{vel conver}) follows directly from uniform estimates in $\eqref{est-U}_{1}$.

\medskip

 From \eqref{pressure est}, we can obtain for $1<r\leq2$,
\begin{equation*}
  \|\tilde{P_\varepsilon}\|_{L^2((0,T)\times\O)}\leq C\|\nabla\tilde{P_\varepsilon}\|_{L^2(0,T;W^{-1,2}(\O))}\leq C,
\end{equation*}
and for $r>2$,
\begin{equation*}
  \|\tilde{P_\varepsilon}\|_{L^{\frac{r}{r-1}}((0,T)\times\O)}\leq C\|\nabla\tilde{P_\varepsilon}\|_{L^{\frac{r}{r-1}}(0,T;W^{-1,{\frac{r}{r-1}}}(\O))}\leq C.
\end{equation*}
Thus we can find
\begin{equation*}
P\in
\begin{cases}
L^2((0,T)\times\Omega) & 1<r\leq2,\\
L^{\frac{r}{r-1}}((0,T)\times\Omega) & r>2,
\end{cases}
\end{equation*}
such that (\ref{p con}) holds.

\medskip

Next we will use the cell problem to construct test functions. Clearly $w^{i,\varepsilon}$ defined in (\ref{2.12.5}) vanishes on the holes in $\Omega_\varepsilon$. Given any scalar function $\phi\in C_c^\infty((0,T)\times\Omega)$, taking $\phi w^{i,\varepsilon}$ as a test function to (\ref{1.11.1})  implies
\begin{align}\label{6.5}
\nonumber
&\int_0^T\int_{\Omega} \tilde{P_\varepsilon}\,{\rm div}\,(\phi w^{i,\varepsilon})\,{\rm d}x{\rm d}t-\int_0^T\int_{\Omega}\big(\frac{\eta_\infty}{2}\nabla\tilde{ U}_\varepsilon+(\eta_0-\eta_\infty) H_\varepsilon\big):\nabla (\phi w^{i,\varepsilon})\,{\rm d}x{\rm d}t\\
&=\int_0^T\int_{\Omega}\big(-F+\varepsilon^2 (\tilde{\mathbf{u}}_\varepsilon-\mathbf{u}_0)+G_\varepsilon
  \big)\cdot\phi w^{i,\varepsilon} \,{\rm d}x{\rm d}t.
\end{align}

Then we will pass $\varepsilon\rightarrow0$ term by term where the limits are taken up to subsequences. It follows from \eqref{w^i con} that
\begin{equation}\label{6.6}
  \lim_{\varepsilon\rightarrow0}\int_0^T\int_{\Omega}{ F}\cdot w^{i,\varepsilon} \phi\, {\rm d}x{\rm d}t=\int_0^T\int_{\Omega}{ F}\cdot \bar{w}^i \phi\, {\rm d}x{\rm d}t.
\end{equation}

The estimates of $\uu_{\e}$ in \eqref{2.1.01} ensure
\begin{equation}\label{6.7}
\big|\int_0^T\int_\Omega \varepsilon^2(\tilde{\mathbf{u}}_\varepsilon-\mathbf{u}_0)\cdot\phi w^{i,\varepsilon} \,{\rm d}x{\rm d}t\big|\leq  \varepsilon^2 \|\tilde{\mathbf{u}}_\varepsilon-\mathbf{u}_0\|_{L^2((0,T)\times \Omega))}\|\phi w^{i,\varepsilon}\|_{L^2((0,T)\times \Omega)}\leq C \e^{2}\to 0.
\end{equation}

For the term related to the pressure, using the divergence free condition ${\rm div}\,w^{i,\varepsilon}=0$ implies
\ba\label{limit-P-1}
  & \int_0^T\int_\Omega\tilde{P_\varepsilon}\,{\rm div}\,(\phi w^{i,\varepsilon})\, {\rm d}x{\rm d}t=\int_0^T\int_\Omega \tilde{P_\varepsilon} w^{i,\varepsilon}\cdot\nabla\phi\,  {\rm d}x{\rm d}t\\
   &=\int_0^T\int_\Omega \tilde{P_\varepsilon}(w^{i,\varepsilon}-\bar{w}^i)\cdot\nabla\phi \,{\rm d}x{\rm d}t+\int_0^T\int_\Omega \tilde{P_\varepsilon}\bar{w}^i\cdot\nabla\phi \, {\rm d}x{\rm d}t.
\ea
By the fact that
\begin{equation*}
\tilde{P_\varepsilon}\rightarrow P \ \rm weakly  \ in
\begin{cases}
L^2((0,T)\times\Omega) & 1<r\leq2,\\
L^{\frac{r}{r-1}}((0,T)\times\Omega) & r>2,
\end{cases}
\end{equation*}
we have
\begin{equation}\label{2.71.1}
  \displaystyle\lim_{\varepsilon\rightarrow0}\int_0^T\int_\Omega \tilde{P}_{\varepsilon}\bar{w}^i \cdot\nabla\phi\,{\rm d}x{\rm d}t =\int_0^T\int_\Omega P\bar{w}^i\cdot\nabla\phi\, {\rm d}x{\rm d}t=\int_0^T\int_\Omega P\,{\rm div}\,(\bar{w}^i \phi)\, {\rm d}x{\rm d}t.
\end{equation}

For each fixed $t\in(0,T)$, by the divergence free condition ${\rm div}\,(w^{i,\varepsilon}-\bar{w}^i)=0$, we have $(w^{i,\varepsilon}-\bar{w}^i)\cdot\nabla\phi \in L_0^{r+2}(\Omega)$. By employing the classical Bogovskii operator $\mathcal{B}$ in domain $\Omega$, we can find $\psi_\varepsilon=\mathcal{B}\big((w^{i,\varepsilon}-\bar{w}^i)\cdot\nabla\phi\big)$ such that
\begin{equation}\label{2.10.1}
 {\rm div}\,\psi_\varepsilon=(w^{i,\varepsilon}-\bar{w}^i)\cdot\nabla\phi, \quad \mbox{for each $t\in(0,T)$.}
\end{equation}
Moreover, we have the following estimate:
\begin{equation}\label{2.102}
  \|\psi_\varepsilon\|_{L^{\infty}(0,T;W_0^{1, r+2 }(\Omega))}\leq C\|(w^{i,\varepsilon}-\bar{w}^i)\cdot\nabla\phi\|_{L^{\infty}(0,T; L^{ r+2 }(\Omega))}\leq C.
\end{equation}
By the fact that $\partial_t\psi_\varepsilon=\partial_t\mathcal{B}\big((w^{i,\varepsilon}-\bar{w}^i)\cdot\nabla\phi\big)=\mathcal{B}\big((w^{i,\varepsilon}-\bar{w}^i)\cdot\partial_t\nabla\phi\big)$
we obtain
\begin{equation*}
  \|\partial_t\psi_\varepsilon\|_{L^{\infty}(0,T;W_0^{1, r+2 }(\Omega))}\leq C\|(w^{i,\varepsilon}-\bar{w}^i)\cdot\partial_t\nabla\phi\|_{L^{\infty}(0,T; L^{ r+2 }(\Omega))}\leq C.
\end{equation*}
By compact Sobolev embedding, we have, up to a subsequence that
\begin{equation}\label{psi str}
  \psi_\varepsilon\rightarrow \psi\quad {\rm strongly \ in \ }L^{ r+2 }((0,T)\times\Omega)
\end{equation}
for some $\psi\in W^{1, r+2 }(0,T;W_0^{1, r+2 }(\Omega))$.
Recall that $w^{i,\varepsilon} \to \bar{w}^i$ weakly in $L^{ r+2 }((0,T)\times\Omega)$. Then
\begin{equation}\label{2.10.5}
  (w^{i,\varepsilon}-\bar{w}^i)\cdot\nabla\phi\rightarrow0\ {\rm weakly \ in} \ L^{ r+2 }((0,T)\times\Omega).
\end{equation}
By (\ref{2.10.1}), (\ref{psi str}) and (\ref{2.10.5}), we can deduce that ${\rm div}\,\psi=0$. Then using (\ref{pressure est}) implies
\ba
\big|\int_0^T\int_\Omega \tilde{P_\varepsilon}(w^{i,\varepsilon}-\bar{w}^i)\cdot\nabla\phi \,{\rm d}x{\rm d}t\big|&=\big|\int_0^T\int_\Omega \tilde{P_\varepsilon}\,{\rm div}\,\psi_\varepsilon\,{\rm d}x{\rm d}t\big|\\
&=\big|\int_0^T\int_\Omega \tilde{P_\varepsilon}\,{\rm div}\,(\psi_\varepsilon-\psi)\,{\rm d}x{\rm d}t\big|\\
&=\big|\langle\nabla \tilde{P_\varepsilon},\psi_\varepsilon-\psi\rangle_{(0,T)\times\Omega}\big|\\
&\leq C
\big(\|\psi_\varepsilon-\psi\|_{L^{ r+2 }((0,T)\times\Omega)}+\e\| \nabla (\psi_\varepsilon-\psi)\|_{L^{ r+2 }((0,T)\times\Omega)}\big).
\nn\ea
Then, together with (\ref{2.102}) and (\ref{psi str}) we finally deduce
\begin{equation}\label{2.82}
  \int_0^T\int_\Omega \tilde{P_\varepsilon}(w^{i,\varepsilon}-\bar{w}^i)\cdot\nabla\phi \,{\rm d}x{\rm d}t\to 0,
\end{equation}
and consequently, by  (\ref{2.71.1}) and (\ref{2.82}), passing $\e\to 0$ in $\eqref{limit-P-1}$ gives
\begin{equation}\label{2.8}
  \displaystyle\lim_{\varepsilon\rightarrow0}\int_0^T\int_\Omega\tilde{P_\varepsilon}\,{\rm div}\,(\phi w^{i,\varepsilon})\, {\rm d}x{\rm d}t=\int_0^T\int_\Omega P\,{\rm div}\,(\phi \bar{w}^i) \,{\rm d}x{\rm d}t.
\end{equation}

By $\eqref{est-U}_{3}$, we have
\begin{equation}\label{2.10}
\big|\int_0^T\int_\Omega G_\varepsilon\cdot\phi w^{i,\varepsilon} \,{\rm d}x{\rm d}t\big|\leq  \|G_\varepsilon\|_{L^{\frac{3}{2}}((0,T)\times\Omega)}\|\phi\|_{L^\infty((0,T)\times\Omega)}\|w^{i,\varepsilon}\|_{L^{\infty}(\Omega)} \leq C \e^{2} \to 0.
\end{equation}

Next we calculate
$$\displaystyle\lim_{\varepsilon\rightarrow0}\int_0^T\int_\Omega\nabla\tilde{U}_\varepsilon:\nabla(\phi w^{i,\varepsilon}) \,{\rm d}x{\rm d}t$$
which equals to
$$\lim_{\varepsilon\rightarrow0}\int_0^T\int_\Omega\nabla\tilde{U}_\varepsilon:(\nabla\phi\otimes w^{i,\varepsilon})\,{\rm d}x{\rm d}t+\lim_{\varepsilon\rightarrow0}\int_0^T\int_\Omega\nabla\tilde{U}_\varepsilon:\nabla w^{i,\varepsilon}\phi \,{\rm d}x{\rm d}t.$$
By $\eqref{est-U}_{1}$ we obtain
\begin{equation}\label{u lim 1}
  \big | \int_0^T\int_\Omega\nabla\tilde{U}_\varepsilon:\nabla\phi\otimes w^{i,\varepsilon}\, {\rm d}x{\rm d}t \big|\leq C \|w^{i,\varepsilon}\|_{L^2(\Omega)}\|\nabla\tilde{U}_\varepsilon\|_{L^2((0,T)\times\Omega)} \leq C \e \to 0.
\end{equation}
Taking $\phi \tilde{U}_\varepsilon$ as a test function to (\ref{2.12.5}) gives
$$\int_0^T\int_\Omega\varepsilon\pi^{i,\varepsilon}\nabla\phi\cdot \tilde{U}_\varepsilon-\varepsilon^2\nabla w^{i,\varepsilon}:\nabla(\phi \tilde{U}_\varepsilon)+e^i\phi \tilde{U}_\varepsilon\, {\rm d}x{\rm d}t=0.$$
It is equivalent to
\begin{equation}\label{2.19}
  \int_0^T\int_\Omega\varepsilon^{-1}\pi^{i,\varepsilon}\nabla\phi\cdot \tilde{U}_\varepsilon-\nabla w^{i,\varepsilon}:\nabla\tilde{U}_\varepsilon\phi-\nabla w^{i,\varepsilon}:(\nabla\phi\otimes \tilde{U}_\varepsilon)+\varepsilon^{-2}\phi \tilde{U}_\varepsilon\cdot e^i\, {\rm d}x{\rm d}t=0.
\end{equation}
 By $\eqref{est-U}_{1}$, (\ref{2.13}), we have
\ba
\big|\int_0^T\int_\Omega\varepsilon^{-1}\pi^{i,\varepsilon}\nabla\phi\cdot \tilde{U}_\varepsilon \,{\rm d}x{\rm d}t\big|\leq \varepsilon^{-1}\|\tilde{U}_\varepsilon\|_{L^2((0,T)\times\Omega)}\|\pi^{i,\varepsilon}\|_{L^2(\Omega)}\|\nabla\phi\|_{L^{\infty}((0,T)\times\O)}
\leq C\e\to0,
\nn\ea
\ba
\big|\int_0^T\int_\Omega \nabla w^{i,\varepsilon}:(\nabla\phi\otimes \tilde{U}_\varepsilon)\,{\rm d}x{\rm d}t\big|  \leq   \|\nabla w^{i,\varepsilon}\|_{L^2(\Omega)}\|\tilde{U}_\varepsilon\|_{L^2((0,T)\times\Omega)}\|\nabla\phi\|_{L^{\infty}((0,T)\times\O)}
 \leq C \e\to 0.
\nn\ea
Then, passing $\varepsilon\rightarrow0$ in (\ref{2.19}) gives
\begin{equation}\label{u lim 2}
  \displaystyle\lim_{\varepsilon\rightarrow0}\int_0^T\int_\Omega \nabla w^{i,\varepsilon}:\nabla\tilde{U}_\varepsilon\phi \,{\rm d}x{\rm d}t=\lim_{\varepsilon\rightarrow0}\int_0^T\int_\Omega \varepsilon^{-2}\phi \tilde{U}_\varepsilon\cdot e^i\, {\rm d}x{\rm d}t=\int_0^T\int_\Omega\phi U_i\,{\rm d}x{\rm d}t.
\end{equation}
Thus by (\ref{u lim 1}) and (\ref{u lim 2}) we have
\begin{equation}\label{6.9}
 \displaystyle\lim_{\varepsilon\rightarrow0}\int_0^T\int_\Omega\nabla{\tilde{U}_\varepsilon}:\nabla(\phi w^{i,\varepsilon}) \,{\rm d}x{\rm d}t=\int_0^T\int_\Omega\phi U_i\,{\rm d}x{\rm d}t.
\end{equation}

\medskip

For the last term related to the nonlinear stress tensor $H_{\e}$,  due to the smallness of $D\tilde{\mathbf{u}}_\varepsilon$, we shall show its contribution in the limit is nothing but a Newtonian stress tensor.  Introduce the decomposition
\ba\label{limit-dec-He}
H_{\e}&=\int_0^t(1+\lambda| D{\mathbf{u}}_\varepsilon|^2)^{\frac{r}{2}-1}D{\mathbf{u}}_\varepsilon\,{\rm d}s= \int_0^t\big((1+ \lambda| D{\mathbf{u}}_\varepsilon|^2)^{\frac{r}{2}-1} - 1\big) D {\mathbf{u}}_\varepsilon \,{\rm d}s+ \int_0^tD {\mathbf{u}}_\varepsilon\,{\rm d}s\\
&=\int_0^t\big((1+ \lambda| D{\mathbf{u}}_\varepsilon|^2)^{\frac{r}{2}-1} - 1\big) D {\mathbf{u}}_\varepsilon \,{\rm d}s+ D {U}_\varepsilon.
\nn\ea
Then
\ba\label{H dec}
&\int_0^T\int_\Omega H_\varepsilon:\nabla(\phi w^{i,\varepsilon})\, {\rm d}x{\rm d}t\\
&=\int_0^T\int_\Omega D\tilde{U}_\varepsilon:\nabla (\phi w^{i,\varepsilon})\, {\rm d}x{\rm d}t+\int_0^T\int_\Omega\left(\int_0^t\big((1+\lambda| D\tilde{\mathbf{u}}_\varepsilon|^2)^{\frac{r}{2}-1}-1\big)D\tilde{\mathbf{u}}_\varepsilon {\rm d}s\right):\nabla (\phi w^{i,\varepsilon})\,{\rm d}x{\rm d}t.
\ea
Using the divergence free condition ${\rm div}\,\tilde{U}_\varepsilon=0$ and (\ref{6.9}) implies
\begin{equation}\label{U limit}
 \lim_{\e\to0}\int_0^T\int_\Omega D\tilde{U}_\varepsilon:\nabla (\phi w^{i,\varepsilon})\, {\rm d}x{\rm d}t=\lim_{\e\to0} \frac{1}{2} \int_0^T\int_\Omega \nabla \tilde{U}_\varepsilon:\nabla (\phi w^{i,\varepsilon})\, {\rm d}x{\rm d}t=\frac{1}{2}\int_0^T\int_\Omega\phi U_i\,{\rm d}x{\rm d}t.
\end{equation}

\medskip

For the other term on the right-hand side of \eqref{H dec}, based on different values of $r$, we use different inequalities to show that its limit  is actually zero.

\medskip

For $1<r<2$, by inequality $0\leq(1+s)^\alpha-s^\alpha\leq1\,(0\leq\alpha\leq1,  \ s \geq0)$, we have
$$\big|(1+\lambda| D\tilde{\mathbf{u}}_\varepsilon|^2)^{\frac{r}{2}-1}-1\big|=\big|(1+\lambda| D\tilde{\mathbf{u}}_\varepsilon|^2)^{\frac{r}{2}-1}\big(1-(1+\lambda| D\tilde{\mathbf{u}}_\varepsilon|^2)^{1-\frac{r}{2}}\big)\big|\leq C|D\tilde{\mathbf{u}}_\varepsilon|^{2-r}.$$
Then, for $1<r<2$, there holds
 \begin{equation*}
\begin{split}
& \big|\int_0^T\int_\Omega\left(\int_0^t\big((1+\lambda| D\tilde{\mathbf{u}}_\varepsilon|^2)^{\frac{r}{2}-1}-1\big)D\tilde{\mathbf{u}}_\varepsilon \,{\rm d}s\right):\nabla w^{i,\varepsilon}\phi\,  {\rm d}x{\rm d}t\,\big|\\
  &\leq C \| \nabla w^{i,\varepsilon} \|_{L^{\infty}(\O)}\int_0^T{\rm d}t\int_0^t{\rm d}s\int_\Omega\mid(1+\lambda| D\tilde{\mathbf{u}}_\varepsilon|^2)^{\frac{r}{2}-1}-1\mid| D\tilde{\mathbf{u}}_\varepsilon|\,{\rm d}x \\
  &\leq C\varepsilon^{-1}\int_0^T{\rm d}t\int_0^t{\rm d}s\int_\Omega| D\tilde{\mathbf{u}}_\varepsilon|^{3-r}{\rm d}x \\
  & \leq C\varepsilon^{-1}\|D\tilde{\mathbf{u}}_\varepsilon\|_{L^{3-r}((0,T)\times\Omega)}^{3-r} \leq C\varepsilon^{-1}\|D\tilde{\mathbf{u}}_\varepsilon\|_{L^{2}((0,T)\times\Omega)}^{3-r}\leq C\varepsilon^{2-r} \to 0.
\end{split}
\end{equation*}

\medskip

 For $2<r\leq4$, again by inequality $0\leq(1+s)^\alpha-s^\alpha\leq1\,(0\leq\alpha\leq1, \ s\geq0)$, we have
$$\big|(1+\lambda| D\tilde{\mathbf{u}}_\varepsilon|^2)^{\frac{r}{2}-1}-1\big|\leq
C|D\tilde{\mathbf{u}}_\varepsilon|^{r-2}.$$
Then using the estimate $\|\nabla \tilde{\mathbf{u}}_\varepsilon\|_{L^{r}((0,T)\times\Omega)}\leq C\e^{\frac 2r}$  in $\eqref{2.1.01}_{2}$ gives
\ba
& \big|\int_0^T\int_\Omega\left(\int_0^t\big((1+\lambda| D\tilde{\mathbf{u}}_\varepsilon|^2)^{\frac{r}{2}-1}-1\big)D\tilde{\mathbf{u}}_\varepsilon \,{\rm
d}s\right):\nabla w^{i,\varepsilon}\phi\,  {\rm d}x{\rm d}t\,\big|\\
  &\leq C\| \nabla w^{i,\varepsilon} \|_{L^{\infty}(\O)} \int_0^T{\rm d}t\int_0^t{\rm d}s\int_\Omega\mid(1+\lambda| D\tilde{\mathbf{u}}_\varepsilon|^2)^{\frac{r}{2}-1}-1\mid|
  D\tilde{\mathbf{u}}_\varepsilon|\,{\rm d}x \\
  &\leq C\varepsilon^{-1}\int_0^T{\rm d}t\int_0^t{\rm d}s\int_\Omega| D\tilde{\mathbf{u}}_\varepsilon|^{r-1}{\rm d}x\\
  &\leq C\varepsilon^{-1}\|D\tilde{\mathbf{u}}_\varepsilon\|_{L^{r-1}((0,T)\times\Omega)}^{r-1} \leq C \e^{-1} \e^{\frac{2(r-1)}{r}} = C \e^{\frac{r-2}{r}} \to 0.
\nn\ea

\medskip

For $r\geq 4$,
\ba
\big|(1+\lambda| D\tilde{\mathbf{u}}_\varepsilon|^2)^{\frac{r}{2}-1}-1\big| \leq\lambda(\frac{r}{2}-1)(1+\lambda| D\tilde{\mathbf{u}}_\varepsilon|^2)^{\frac{r}{2}-2}|D\tilde{\mathbf{u}}_\varepsilon|^{2} \leq C(|D\tilde{\mathbf{u}}_\varepsilon|^{2}+|D\tilde{\mathbf{u}}_\varepsilon|^{r-2}).
\nn\ea
Using the estimate $\|\nabla \tilde{\mathbf{u}}_\varepsilon\|_{L^{2}((0,T)\times\Omega)}^{2} + \|\nabla \tilde{\mathbf{u}}_\varepsilon\|_{L^{r}((0,T)\times\Omega)}^{r}\leq C\e^2$ and H\"older's inequality implies
\ba
\|\nabla \tilde{\mathbf{u}}_\varepsilon\|_{L^{q}((0,T)\times\Omega)}^{q}  \leq C \e^{2}, \quad \forall \, q\in [2,r].
\nn\ea
Thus,
\ba
& \big|\int_0^T\int_\Omega\left(\int_0^t\big((1+\lambda| D\tilde{\mathbf{u}}_\varepsilon|^2)^{\frac{r}{2}-1}-1\big)D\tilde{\mathbf{u}}_\varepsilon \,{\rm
d}s\right):\nabla w^{i,\varepsilon}\phi\,  {\rm d}x{\rm d}t\,\big|\\
  &\leq C \| \nabla w^{i,\varepsilon} \|_{L^{\infty}(\O)}  \int_0^T{\rm d}t\int_0^t{\rm d}s\int_\Omega\mid(1+\lambda| D\tilde{\mathbf{u}}_\varepsilon|^2)^{\frac{r}{2}-1}-1\mid|
  D\tilde{\mathbf{u}}_\varepsilon|\,{\rm d}x \\
  &\leq C\varepsilon^{-1}\int_0^T{\rm d}t\int_0^t{\rm d}s\int_\Omega \big( | D\tilde{\mathbf{u}}_\varepsilon|^{r-1}+| D\tilde{\mathbf{u}}_\varepsilon|^{3} \big){\rm d}x\\
  &\leq C\varepsilon^{-1}(\|D\tilde{\mathbf{u}}_\varepsilon\|_{L^{r-1}((0,T)\times\Omega)}^{r-1}+\|D\tilde{\mathbf{u}}_\varepsilon\|_{L^{3}((0,T)\times\Omega)}^{3})\leq C \e \to 0.
  \nn\ea

  \medskip

To sum up, for $1<r<\infty$,
\begin{equation}\label{str con 1}
  \int_0^T\int_\Omega\left(\int_0^t\big((1+\lambda| D\tilde{\mathbf{u}}_\varepsilon|^2)^{\frac{r}{2}-1}-1\big)D\tilde{\mathbf{u}}_\varepsilon \,{\rm d}s\right):\nabla w^{i,\varepsilon}\phi\, {\rm d}x{\rm d}t\to0.
\end{equation}
We can use same arguments to show that
\begin{equation}\label{str con 2}
 \left | \int_0^T\int_\Omega\left(\int_0^t\big((1+\lambda| D\tilde{\mathbf{u}}_\varepsilon|^2)^{\frac{r}{2}-1}-1\big)D\tilde{\mathbf{u}}_\varepsilon \,{\rm d}s\right):(w^{i,\varepsilon}\otimes\nabla\phi)\, {\rm d}x{\rm d}t \right|\leq C \e \to 0.
\end{equation}
Thus by (\ref{str con 1}) and (\ref{str con 2}), we have
\begin{equation}\label{3.28.1}
  \int_0^T\int_\Omega\left(\int_0^t\big((1+\lambda| D\tilde{\mathbf{u}}_\varepsilon|^2)^{\frac{r}{2}-1}-1\big)D\tilde{\mathbf{u}}_\varepsilon {\rm d}s\right):\nabla (\phi w^{i,\varepsilon})\,{\rm d}x{\rm d}t\to0.
\end{equation}
Then using (\ref{U limit}), (\ref{3.28.1}) and passing $\varepsilon\rightarrow0$ in (\ref{H dec}) implies
\begin{equation}\label{2.11}
  \displaystyle\lim_{\varepsilon\rightarrow0}\int_0^T\int_\Omega H_\varepsilon:\nabla(\phi w^{i,\varepsilon})\, {\rm d}x{\rm d}t=\frac{1}{2}\int_0^T\int_\Omega\phi U_i\,{\rm d}x{\rm d}t.
\end{equation}

Finally using (\ref{6.6}), (\ref{6.7}), (\ref{2.8}), (\ref{2.10}), (\ref{6.9}), (\ref{2.11}) and passing $\varepsilon\rightarrow0$ in (\ref{6.5}) gives
$$\int_0^T\int_\Omega \frac{1}{2}\eta_0\phi U_i\,{\rm d}x{\rm d}t=\int_0^T\int_\Omega F\cdot\phi \bar{w}^i\,{\rm d}x{\rm d}t+\int_0^T\int_\Omega P\,{\rm div}\,(\phi \bar{w}^i)\,{\rm d}x{\rm d}t.$$
This is nothing but the Darcy's law \eqref{Darcy-U} in the sense of distribution, with permeability tensor $A=(\bar{w}_j^i)$ determined in (\ref{1.21.0}). Thus, we completed the proof of Proposition \ref{pro2.2}.
\end{proof}
\end{proposition}

\subsection{End of the proof}
From Proposition \ref{pro2.2}, we can deduce the limit equation in ${\mathbf u}$. Set $\tilde{p}_\varepsilon=\partial_t\tilde{P}_\varepsilon $ and  $p=\partial_tP$ in the weak sense.  By (\ref{p def}) and (\ref{p con}) we have
\begin{equation*}
p\in
\begin{cases}
W^{-1,2}(0,T;L^2(\Omega)) & 1<r\leq2,\\
W^{-1,\frac{r}{r-1}}(0,T;L^{\frac{r}{r-1}}(\Omega))& r>2,
\end{cases}
\end{equation*}
and
\begin{equation*}
\tilde{p}_\varepsilon=\partial_t\tilde{P}_\varepsilon\rightarrow p\ {\rm weakly  \ in}
\begin{cases}
W^{-1,2}(0,T;L^2(\Omega)) & 1<r\leq2,\\
W^{-1,\frac{r}{r-1}}(0,T;L^{\frac{r}{r-1}}(\Omega))& r>2.
\end{cases}
\end{equation*}
Thus, differentiating (\ref{Darcy-U}) with respect to time variable in the distribution sense gives the limit equation \eqref{1.10.1}. The proof of Theorem \ref{thm-1} is completed.

\section*{Acknowledgements}
Both authors are partially supported by the NSF of China under Grant 12171235.


\end{document}